%% file: boldface_resurrection.tex
\documentclass[12pt]{amsart}

\title[Strongly uplifting cardinals and boldface resurrection]{Strongly uplifting cardinals and the boldface resurrection axioms}

\author[Hamkins]{Joel David Hamkins}
 \address[J.~D.~Hamkins]
        {Mathematics, Philosophy, Computer Science,
          The Graduate Center of The City University of New York,
          365 Fifth Avenue, New York, NY 10016
          \&
          Mathematics,
          College of Staten Island of CUNY,
          Staten Island, NY 10314}
\email{jhamkins@gc.cuny.edu}
\urladdr{http://jdh.hamkins.org}

\author[Johnstone]{Thomas A. Johnstone}
 \address[T. A. Johnstone]
         {Mathematics,
          New York City College of Technology of CUNY,
          NY 11201}
 \email{tjohnstone@citytech.cuny.edu}
 \urladdr{http://websupport1.citytech.cuny.edu/faculty/tjohnstone/}

\thanks{This article is a successor to~\cite{HamkinsJohnstone2014:ResurrectionAxiomsAndUpliftingCardinals}. The research of the first author has been supported in part by NSF grant DMS-0800762, PSC-CUNY grant 64732-00-42 and  Simons Foundation grant 209252, and the authors together are supported by grant 80209-06 20 from the CUNY Collaborative Incentive Award program. The research of the second author has been supported by a CUNY Scholar Incentive Award, PSC-CUNY grants 62803-00-40 and 64682-00-42, and by grants P20835-N13 and P21968-N13 from the Austrian Science Fund (FWF) during his 2009-2010 visit to the Kurt G\"odel Research Center at the University of Vienna. Commentary concerning this article can be made at http://jdh.hamkins.org/strongly-uplifting-cardinals-and-boldface-resurrection.}

\usepackage[hidelinks]{hyperref}
\usepackage{latexsym,amsfonts,amsmath,amssymb}
\usepackage{verbatim}
\usepackage{diagrams}\diagramstyle[tight,centredisplay]
\include{MathMacrosJDH}

\newcommand{\RAbfall}{\RAbf(\text{all})}

\newcommand{\smallgt}{\mathrel{\mathchoice{\raise2pt\hbox{$\scriptstyle>$}}{\raise1pt\hbox{$\scriptstyle>$}}{\raise0pt\hbox{$\scriptscriptstyle>$}}{\scriptscriptstyle>}}}

\newcommand{\ltjkappa}{{{\smalllt}j(\kappa)}}

\newcommand{\smallgeq}{\mathrel{\mathchoice{\raise2pt\hbox{$\scriptstyle\geq$}}{\raise1pt\hbox{$\scriptstyle\geq$}}{\raise1pt\hbox{$\scriptscriptstyle\geq$}}{\scriptscriptstyle\geq}}}

\newcommand{\wRA}{{\rm wRA}}
\newcommand{\RAbf}{{\UnderTilde\RA}}
\newcommand{\wRAbf}{{\rm w}\RAbf}
\newcommand{\proper}{\text{proper}}
\newcommand{\semiproper}{\text{semi-proper}}
\newcommand{\Hc}{H_\continuum}

\newcommand{\BFA}{\rm BFA}

\begin{document}

\begin{abstract}
 We introduce the strongly uplifting cardinals, which are equivalently characterized, we prove, as the superstrongly unfoldable cardinals and also as the almost-hugely unfoldable cardinals, and we show that their existence is equiconsistent over \ZFC\ with natural instances of the boldface resurrection axiom, such as the boldface resurrection axiom for proper forcing.
\end{abstract}

\maketitle

\section{Introduction}

The strongly uplifting cardinals, which we shall introduce in this article, are a boldface analogue of the uplifting cardinals of~\cite{HamkinsJohnstone2014:ResurrectionAxiomsAndUpliftingCardinals}, and are equivalently characterized as the superstrongly unfoldable cardinals and also as the almost-hugely unfoldable cardinals. In consistency strength, these new large cardinals lie strictly above the weakly compact, totally indescribable and strongly unfoldable cardinals and strictly below the subtle cardinals, which in turn are weaker in consistency than the existence of $0^\sharp$. The robust diversity of equivalent characterizations of this new large cardinal concept enables constructions and techniques from much larger large cardinal contexts, such as Laver functions and forcing iterations with applications to forcing axioms. Using such methods, we prove that the existence of a strongly uplifting cardinal (and hence also a superstrongly unfoldable or almost-hugely unfoldable cardinal) is equiconsistent over \ZFC\ with natural instances of the boldface resurrection axioms, including the boldface resurrection axiom for proper forcing, for semi-proper forcing, for c.c.c.~forcing and others. Thus, whereas in~\cite{HamkinsJohnstone2014:ResurrectionAxiomsAndUpliftingCardinals} we proved that the existence of a mere uplifting cardinal is equiconsistent with natural instances of the (lightface) resurrection axioms, here we adapt both of these notions to the boldface context.

These forcing arguments, we believe, evoke the essential nature of Baumgartner's seminal argument forcing \PFA\ from a supercompact cardinal, and so we are honored and pleased to be a part of this memorial issue in honor of James Baumgartner.

%%%%%%%%%%%%%%%%%%%%%%%%%%%%%%%%%%%%%%%%%%%%%%%%%%%%%%%%%%%%%%%%%%%%%%%%%%%%%%%%%%%%%%%%%%%%%
%%%%%%%%%%%%%% section:         THE STRONGLY UPLIFTING CARDINALS  %%%%%%%%%%%%%%%%%%%%%%%%%%% %%%%%%%%%%%%%%%%%%%%%%%%%%%%%%%%%%%%%%%%%%%%%%%%%%%%%%%%%%%%%%%%%%%%%%%%%%%%%%%%%%%%%%%%%%%%%

\section{Strongly uplifting, superstrongly unfoldable and almost-hugely unfoldable cardinals}\label{S.StrongUplift}

Let us now introduce the strongly uplifting cardinals, which strengthen the uplifting cardinal concept from~\cite{HamkinsJohnstone2014:ResurrectionAxiomsAndUpliftingCardinals} by the involvement of the predicate parameter $A$, allowing us to view the strong uplifting property as a boldface form of upliftingness.

\begin{definition}\label{Definition.ThetaStrUplift}
 \rm An inaccessible cardinal $\kappa$ is {\df strongly uplifting} if it is strongly $\theta$-uplifting for every ordinal $\theta$, which is to say that for every $A\of V_\kappa$ there is an inaccessible cardinal $\gamma\geq\theta$ and a set $A^*\of V_\gamma$ such that $\<V_\kappa,{\in},A>\elesub\<V_\gamma,{\in},A^*>$ is a proper elementary extension.
\end{definition}

This definition generalizes the concept of $\kappa$ being {\df uplifting}, which is simply the case where $A$ is trivial or omitted~\cite{HamkinsJohnstone2014:ResurrectionAxiomsAndUpliftingCardinals}. It would be equivalent to require the property only for $A\of\kappa$, since any predicate on $V_\kappa$ can be easily coded with a subset of $\kappa$, and we shall henceforth often adopt this perspective. Further, we needn't actually require here that $\kappa$ is inaccessible at the outset, but only an ordinal, since the inaccessibility of $\kappa$ and much more follows from the extension property itself (using just subsets $A\of\kappa$), as we explain in the proof of theorem~\ref{T.ExtensionCharacterizations}. It is immediate from the definition that every strongly uplifting cardinal is strongly unfoldable (and hence also weakly compact, totally indescribable and so on), since by the extension characterization of strong unfoldability (see~\cite{Villaveces1998:ChainsOfEndElementaryExtensionsOfModelsOfSetTheory,Villaveces&Leshem1999:FailureOfGCHatUnfoldableCardinals, Hamkins2001:UnfoldableCardinals}),  an inaccessible cardinal $\kappa$ is strongly unfoldable just in case for every ordinal $\theta$ and every $A\of\kappa$ there is $A^*$ and transitive set $W$ with $V_\theta\of W$, such that $\<V_\kappa,{\in},A>\elesub\<W,{\in},A^*>$. The strongly uplifting cardinals strengthen this by insisting that $W$ has the form $V_\gamma$ for some inaccessible cardinal $\gamma$. So strong unfoldability is a lower bound for strong upliftingness, and more refined lower bounds are provided by theorem~\ref{T.ConsistLowerBound}. For a crude upper bound, it is clear that if $\kappa$ is {\df super $1$-extendible}, which means that there are arbitrarily large $\theta$ for which there is an elementary embedding $j:V_{\kappa+1}\to V_{\theta+1}$, then $\kappa$ is also strongly uplifting, simply by letting $A^*=j(A)$ for any particular $A\of V_\kappa$. An improved upper bound in consistency strength is provided by the observation (theorem~\ref{T.ZeroSharp}) that if $0^\sharp$ exists, then every Silver indiscernible is strongly uplifting in $L$; a still lower upper bound is provided by the subtle cardinals in theorem~\ref{T.SubtleUpperBound}. Meanwhile, let's show that the strongly uplifting cardinals are downward absolute to the constructible universe $L$.

\begin{theorem}\label{Theorem.ThetaStrUpliftDwnAbsL}
 Every strongly uplifting cardinal is strongly uplifting in~$L$. Indeed, every strongly $\theta$-uplifting cardinal is strongly $\theta$-uplifting in~$L$.
\end{theorem}

\begin{proof}
Suppose that $\kappa$ is strongly $\theta$-uplifting in $V$. Since $\kappa$ is inaccessible, it is also inaccessible in $L$. Consider any set $A\of L_\kappa=V_\kappa^L$ with $A\in L$. Since $A$ is constructible, it must be that $A\in L_\beta$ for some $\beta<(\kappa^+)^L$. Let $E$ be a relation on $\kappa$ such that $\<\kappa,E>\iso\<L_\beta,{\in}>$. Since $\kappa$ is strongly $\theta$-uplifting in $V$, there is a proper elementary extension $\<V_\kappa,{\in},E>\elesub \<V_\gamma,{\in},E^*>$ for some inaccessible cardinal $\gamma\geq\theta$ and binary relation $E^*$ on $\gamma$. Since $E$ is well-founded, there are no infinite $E$-descending sequences in $V_\kappa$. Since $\gamma$ is regular and $V_\gamma$ is consequently closed under countable sequences, it follows by elementarity that $E^*$ is also well-founded. Further, since $\<V_\kappa,{\in},E>$ can verify that $\<\kappa,E>\satisfies V=L$, it follows by elementarity that $\<\gamma,E^*>$ also satisfies $V=L$, and since it is well-founded it must be that $\<\gamma,E^*>\iso \<L_{\beta^*},{\in}>$ for some ordinal $\beta^*$. Note that $A$ is a class in $\<V_\kappa,{\in},E>$ that is definable from parameters, since $A$ is represented by some ordinal $\alpha<\kappa$ in the structure $\<\kappa,E>$. If $A^*$ is the element of $L_{\beta^*}$ represented by the same $\alpha$ with respect to $E^*$, then it follows by elementarity that $\<L_\kappa,{\in},A>\prec\<L_\gamma,{\in},A^*>$, and since $A^*\in L$, we have witnessed the desired instance of strong $\theta$-uplifting.
\end{proof}

Recall from~\cite{HamkinsJohnstone2014:ResurrectionAxiomsAndUpliftingCardinals} that an inaccessible cardinal $\kappa$ is pseudo uplifting if for every ordinal $\theta$ there is some ordinal $\gamma\geq\theta$, not necessarily inaccessible, for which $V_\kappa\elesub V_\gamma$. Thus, the pseudo-uplifting property simply drops the requirement that the extension height $\gamma$ is inaccessible, and we observed in~\cite[thm~11]{HamkinsJohnstone2014:ResurrectionAxiomsAndUpliftingCardinals} that this change results in a strictly weaker notion. In the boldface context, it is tempting to define similarly that an ordinal $\kappa$ is {\df strongly pseudo uplifting} if for every ordinal $\theta$ it is strongly pseudo $\theta$-uplifting, meaning that for every $A\of \kappa$, there is an ordinal $\gamma\geq\theta$, not necessarily inaccessible, and a set $A^*\of \gamma$ for which $\<V_\kappa,{\in},A>\elesub\<V_\gamma,{\in},A^*>$. Similarly, in the other direction, we might want to define that $\kappa$ is strongly uplifting {\df with weakly compact targets}, if the corresponding extensions $\<V_\kappa,{\in},A>\elesub \<V_\gamma,{\in},A^*>$ can be found where $\gamma$ is weakly compact in $V$. In the boldface context, however, these changes do not actually result in different large cardinal concepts, for we shall presently show that it is equivalent to require nothing extra about the extension height $\gamma$, or to require that it is inaccessible, weakly compact, totally indescribable or much more.

\begin{theorem}[Extension characterizations]\label{T.ExtensionCharacterizations}
 A cardinal is strongly uplifting if and only if it is strongly pseudo uplifting, if and only if it is strongly uplifting with weakly compact targets. Indeed, for any ordinals $\kappa$ and $\theta$, the following are equivalent.
\begin{enumerate}
\item $\kappa$ is strongly pseudo $(\theta+1)$-uplifting. That is, $\kappa$ is an ordinal and for every $A\of \kappa$ there is an ordinal $\gamma>\theta$ and a set $A^*\of \gamma$ such that  $\<V_\kappa,{\in},A>\elesub\<V_\gamma,{\in},A^*>$ is a proper elementary extension.
\item $\kappa$ is strongly $(\theta+1)$-uplifting. That is, $\kappa$ is inaccessible and for every $A\of\kappa$ there is an inaccessible $\gamma>\theta$ and a set $A^*\of \gamma$ such that \hbox{$\<V_\kappa,{\in},A>\elesub\<V_\gamma,{\in},A^*>$} is a proper elementary extension.
\item $\kappa$ is strongly $(\theta+1)$-uplifting with weakly compact targets. That is, $\kappa$ is inaccessible and for every $A\of\kappa$ there is a weakly compact $\gamma$ and $A^*\of\gamma$ such that $\<V_\kappa,{\in},A>\elesub\<V_\gamma,{\in},A^*>$ is a proper elementary extension.
\item $\kappa$ is strongly $(\theta+1)$-uplifting with totally indescribable targets, and indeed with targets having any property of $\kappa$ that is absolute to all models $V_\gamma$ with $\gamma>\kappa,\theta$.
\end{enumerate}
\end{theorem}

\begin{proof} It is clear that $(4)\implies(3)\implies(2)\implies(1)$. Conversely, suppose that statement (1) holds. It is an easy exercise to see that $\kappa$ must be an inaccessible cardinal. Namely, $\kappa$ must be regular, for otherwise we may not have $\<V_\kappa,{\in},A>\elesub\<V_\gamma,{\in},A^*>$ when $A\of\kappa$ is a short cofinal subset of $\kappa$---meaning that the order type is less than $\kappa$---since this order type would be definable in the former structure, but different from the strictly larger corresponding order type of $A^*$ in the second structure, violating elementarity. Similarly, $\kappa$ must be a strong limit cardinal if there is any proper elementary extension $V_\kappa\elesub V_\gamma$ at all, or indeed any transitive extension $V_\kappa\elesub W$ with $\kappa\in W$, since otherwise we could find for some $\beta<\kappa$ a well-ordering of a subset of $P(\beta)$ having order-type exactly $\kappa$, and this order would be an element of $V_\kappa$ having no isomorphism with an ordinal in $V_\kappa$, but it would have such an isomorphism to an ordinal in $W$. So $\kappa$ must be inaccessible. Consider now any set $A\of\kappa$. Let $C\of\kappa$ be the club of ordinals $\delta<\kappa$ for which $\<V_\delta,{\in},A\intersect\delta>\elesub\<V_\kappa,{\in},A>$. Now, consider any proper extension $\<V_\kappa,{\in},A,C>\elesub\<V_\gamma,{\in},A^*,C^*>$, where $\gamma>\theta$, but $\gamma$ is not necessarily inaccessible. Because every element $\delta\in C$ has $\<V_\delta,{\in},A\intersect\delta>\elesub\<V_\kappa,{\in},A>$, it follows by elementarity that $\<V_\eta,{\in},A^*\intersect\eta>\elesub\<V_\gamma,{\in},A^*>$ for every $\eta\in C^*$. Since $\kappa\in C^*$ and $\kappa$ is inaccessible, it follows from $\<V_\kappa,{\in},C>\elesub \<V_\gamma,{\in},C^*>$ that the inaccessible cardinals in $C$ cannot be bounded below $\kappa$, and so by elementarity there must also be unboundedly many inaccessible $\eta\in C^*$. Fix some such inaccessible cardinal $\eta\in C^*$ above $\theta$ and $\kappa$. Combining the information, it follows that $\<V_\kappa,{\in},A>\elesub\<V_\eta,{\in},A^*\intersect\eta>$, and so we've witnessed (2) using the inaccessible cardinal $\eta$. Further, since $\kappa$ is weakly compact in $V_\gamma$, we also could find weakly compact $\eta\in C^*$ above $\theta$ and thereby verify statement (3). Similarly, since $\kappa$ is totally indescribable and much more that is witnessed in $V_\gamma$, strongly unfoldable up to $\gamma$ and so on, we may find corresponding $\eta$ in $C^*$ above $\theta$ and thereby witness statement (4). Namely, for any property of $\kappa$ in $V_\gamma$, we may find $\eta$ with this property in $V_\gamma$ for which $\<V_\kappa,{\in},A>\elesub\<V_\eta,{\in},A^*>$, since there will be unboundedly many such $\eta$ in the club $C^*$.
\end{proof}

One may generally use $H_\kappa$ instead of $V_\kappa$ in the characterizations, provided $\kappa$ is a cardinal. For example, $\kappa$ is strongly uplifting just in case it is a cardinal and for all $A\of \kappa$ there are arbitrarily large cardinals $\gamma$ with sets $A^*\of \gamma$ such that $\<H_\kappa,{\in},A>\elesub\<H_\gamma,{\in},A^*>$, and one may freely assume or not that $\gamma$ is inaccessible, weakly compact, totally indescribable and much more. Note that this boldface extension property for $H_\kappa$ implies that $\kappa$ is inaccessible: it is regular, as before, by using a short cofinal set $A\of\kappa$; and it is a strong limit, since if $2^\beta\geq\kappa$ for some $\beta<\kappa$, then we may divide $\kappa$ into $\kappa$ many interval blocks of size $\beta$ and let $A\of\kappa$ have a different subset of $\beta$ pattern on each block; if $\<H_\kappa,{\in},A>\elesub\<H_\gamma,{\in},A^*>$, then whichever subset of $\beta$ appears in $A^*$ on the block $[\kappa,\kappa+\beta)$ will not appear on $A$ at all, since the patterns on $A^*$ do not repeat, but this pattern is in $H_\kappa$ and does appear on $A^*$, violating elementarity. Note also that the properties in statement (4) include all $\Sigma_2$ properties of $\kappa$ that are realized in the relevant corresponding extensions $V_\gamma$.

We should like now to provide a number of embedding characterizations of the strongly uplifting property. These characterizations will continue the progression of embedding characterizations of the weakly compact cardinals, the indescribable cardinals, the unfoldable cardinals and the strongly unfoldable cardinals. Specifically, if $\kappa$ is any inaccessible cardinal and $\theta$ is any ordinal, then it is known that:
\begin{enumerate}
\item $\kappa$ is weakly compact if and only if for each $A\in H_{\kappa^\plus}$ there is a $\kappa$-model $M\satisfies\ZFC$ with $A\in M$ and a transitive set $N$ with an elementary embedding $j:M\to N$ with critical point $\kappa$.
\item $\kappa$ is $\theta$-unfoldable if and only if for each $A\in H_{\kappa^\plus}$ there is a $\kappa$-model $M\satisfies\ZFC$ with $A\in M$ and a transitive set $N$ with an elementary embedding $j:M\to N$ with critical point $\kappa$ and $j(\kappa)\geq\theta$.
\item$\kappa$ is strongly $\theta$-unfoldable if and only if for each $A\in H_{\kappa^\plus}$ there is a $\kappa$-model $M\satisfies\ZFC$ with $A\in M$ and a transitive set $N$ with an elementary embedding $j:M\to N$ with critical point $\kappa$ and $j(\kappa)\geq\theta$ and $V_\theta\of N$.
\end{enumerate}
For further details, see~\cite{HamkinsJohnstone2010:IndestructibleStrongUnfoldability}, and also~\cite{Villaveces1998:ChainsOfEndElementaryExtensionsOfModelsOfSetTheory, Villaveces&Leshem1999:FailureOfGCHatUnfoldableCardinals},~\cite{Hamkins2001:UnfoldableCardinals},~\cite{Johnstone2007:Dissertation, Johnstone2008:StronglyUnfoldableCardinalsMadeIndestructible},~\cite{Hamkins:ForcingAndLargeCardinals}. A {\df $\kappa$-model} is a transitive set $M$ of size $\kappa$ with $\kappa\in M$ and $M^\ltkappa\of M$, and satisfying the theory $\ZFC^-$, meaning \ZFC\ without power set\footnote{The theory $\ZFC^-$ should be axiomatized with the collection axiom and not merely the replacement axiom (and the axiom of choice should be taken as the well-order principle), especially as here in the context of ultrapower and extender embeddings, for reasons explored in detail in~\cite{GitmanHamkinsJohnstone:WhatIsTheTheoryZFC-Powerset?}, which shows that many expected results, including the \Los\ theorem, do not hold under the naive axiomatization, which is not equivalent to the correct formulation of $\ZFC^-$ axiomatization when one lacks the power set axiom.}\!, although the embedding characterizations above use full \ZFC. One often sees such embedding characterizations using only $M\satisfies\ZFC^-$, but it is equivalent to require full \ZFC\ as we have, as in footnote~\ref{F.specialkappamodel}, since if merely $M\satisfies\ZFC^-$, but $j:M\to N$ is elementary, then $M'=V_{j(\kappa)}^N$ would be a model of full \ZFC\ containing $A$, which could then be used with an embedding $j_1:M'\to N'$. These embedding characterizations are extremely robust, and they remain equivalent characterizations of these large cardinal notions even after diverse minor changes. For example,  one may consider only $A\of\kappa$ rather than $A\in H_{\kappa^+}$; one may add the requirement that $V_\kappa\elesub M$ holds for the $\kappa$-models $M$; there is no need to require $M\satisfies\ZFC$ or even $M\satisfies\ZFC^-$, as any transitive set will do;  one may drop the $M^\ltkappa\of M$ requirement and replace it by $2^{\lt\kappa}=\kappa$\,;  one gets embeddings $j:M\to N$ for every transitive structure of size $\kappa$; by composing embeddings, one may insist that $j(\kappa)>\theta$ and so on.

Just as strong unfoldability is a strong-cardinal analogue of unfoldability, it is natural to consider the corresponding superstrong and almost hugeness analogues of that notion.

\begin{definition}\rm\
 \begin{enumerate}
  \item An inaccessible cardinal $\kappa$ is {\df superstrongly unfoldable}, if for every ordinal $\theta$ it is {\df superstrongly $\theta$-unfoldable}, which is to say that for each $A\in H_{\kappa^\plus}$ there is a $\kappa$-model $M\satisfies\ZFC$ with $A\in M$ and a transitive set $N$ with an elementary embedding $j:M\to N$ with critical point $\kappa$ and $j(\kappa)\geq\theta$ and $V_{j(\kappa)}\of N$.
  \item An inaccessible cardinal $\kappa$ is {\df almost-hugely unfoldable}, if for every ordinal $\theta$ it is {\df almost-hugely $\theta$-unfoldable}, which is to say that for each $A\in H_{\kappa^\plus}$ there is a $\kappa$-model $M\satisfies\ZFC$ with $A\in M$ and a transitive set $N$ with an elementary embedding $j:M\to N$ with critical point $\kappa$ and $j(\kappa)\geq\theta$ and $N^{<j(\kappa)}\of N$.
 \end{enumerate}
\end{definition}

\noindent We needn't insist that $\kappa$ is inaccessible at the outset, since this follows from the properties in question. A natural weakening of these notions does not insist that one may find arbitrarily large such targets $j(\kappa)$, but only one. Namely, a cardinal $\kappa$ is {\df weakly superstrong}, if for every $A\in H_{\kappa^+}$ there is a $\kappa$-model $M\satisfies\ZFC$ with $A\in M$ and an elementary embedding $j:M\to N$ into a transitive set $N$ with critical point $\kappa$ and $V_{j(\kappa)}\of N$. And similarly, $\kappa$ is {\df weakly almost huge}, if for every $A\in H_{\kappa^+}$ there is such $j:M\to N$ with $N^{<j(\kappa)}\of N$.

Remarkably, the superstrongly unfoldable cardinals are precisely the same as the almost-hugely unfoldable cardinals, which are precisely the same as the strongly uplifting cardinals. This phenomenon can be viewed as an extension of the fact pointed out by Hamkins and Dzamonja~\cite{DzamonjaHamkins2006:DiamondCanFail}, that the strongly unfoldable cardinals are equivalently characterized both in terms of strongness type embeddings $j:M\to N$ with $V_\theta\of N$, and also in terms of supercompactness type embeddings $j:M\to N$ with $N^\theta\of N$. Similarly, here, we have strong upliftness characterized both in terms of superstrongness type embeddings $j:M\to N$ with $V_{j(\kappa)}\of N$ and also equivalently in terms of almost hugeness embeddings $j:M\to N$, with $N^{<j(\kappa)}\of N$.

\begin{theorem}[Embedding characterizations]\label{T.EmbeddingCharacterizations} A cardinal is strongly uplifting if and only if it is superstrongly unfoldable. Indeed, for any cardinal $\kappa$ and ordinal $\theta$, the following are equivalent.
 \begin{enumerate}
  \item $\kappa$ is strongly $(\theta+1)$-uplifting.
  \item $\kappa$ is superstrongly $(\theta+1)$-unfoldable.
  \item $\kappa$ is almost-hugely $(\theta+1)$-unfoldable.
  \item  For every set $A\in H_{\kappa^\plus}$ there is a $\kappa$-model $M\satisfies\ZFC$ with $A\in M$ and $V_\kappa\elesub M$ and a transitive set $N$ with an elementary embedding $j:M\to N$ having critical point $\kappa$ with $j(\kappa)>\theta$ and $V_{j(\kappa)}\elesub N$, such that $N^{<j(\kappa)}\of N$ and $j(\kappa)$ is inaccessible, weakly compact and more in $V$.
  \item $\kappa^\ltkappa=\kappa$ holds, and for every $\kappa$-model $M$ there is an elementary embedding $j:M\to N$ having critical point $\kappa$ with $j(\kappa)>\theta$ and $V_{j(\kappa)}\of N$, such that $N^{<j(\kappa)}\of N$ and $j(\kappa)$ is inaccessible, weakly compact and more in $V$.
 \end{enumerate}
\end{theorem}

\begin{proof}
($1\implies 4$) Suppose that $\kappa$ is strongly $(\theta+1)$-uplifting, and consider any set $A\in H_{\kappa^\plus}$. Since $\kappa$ is weakly compact, there is a $\kappa$-model $M\satisfies\ZFC$ with $A\in M$ and $V_\kappa\elesub M$.\footnote{Code the set $A$ by a set $\tilde{A}\of\kappa$ and find first a $\kappa$-model $M'$ with $\tilde{A}\in M'$; use the weak compactness of $\kappa$ to find a transitive set $N$ with an elementary embedding $j:M'\to N$ with critical point $\kappa$, and by using the induced factor embedding, if necessary, assume that $N$ has size $\kappa$ and $N^{\lt\kappa}\of N$. The set $M=(V_{j(\kappa)})^N$ is then the desired $\kappa$-model satisfying $\ZFC$ with $V_\kappa\elesub M$ and $A\in M$.\label{F.specialkappamodel}} Since this structure has size~$\kappa$, we may find a well-founded relation $E$ on $\kappa$ and an isomorphism $\pi:\<M,{\in}>\iso\<\kappa,E>$. By our assumption on $\kappa$, there is an inaccessible cardinal $\gamma$ above $\theta$ and some $E^*\of\gamma$ such that $\<V_\kappa,{\in},E>\elesub \<V_\gamma,{\in},E^*>$, and by theorem~\ref{T.ExtensionCharacterizations} we may also assume that $\gamma$ is weakly compact, totally indescribable and indeed much more. Since $E$ is well-founded, it follows by elementarity that $E^*$ has no infinite descending sequences in $V_\gamma$, and since $\gamma$ is regular, this means that $E^*$ is really well-founded. Let $\tau:\<\gamma,E^*>\to\<N,{\in}>$ be the transitive collapse of $E^*$ onto a transitive set $N$, and let $j=\tau\circ\pi$ be the composition map, so that $j:M\to N$ is an elementary embedding with $A\in M$. Note that $j$ fixes ordinals below $\kappa$, because if $\alpha$ is coded by $\xi$  with respect to $E$, then it is also coded by $\xi$ with respect to $E^*$, and so $j(\alpha)=\alpha$. If $\kappa$ is represented by $\alpha$ with respect to $E$, then $\gamma$ will be represented by $\alpha$ with respect to $E^*$, since this property is expressible in $\<V_\kappa,{\in},E>\elesub\<V_\gamma,{\in},E^*>$, and so $j(\kappa)=\tau(\pi(\kappa))=\tau(\alpha)=\gamma$. Thus, the map $j$ has critical point $\kappa$, with $j(\kappa)=\gamma$ being an inaccessible cardinal above $\theta$. Since the structure $\<V_\kappa,{\in},E>$ sees that each of its elements is coded by an ordinal via $E$, it follows by elementarity that each of the elements of $V_\gamma$ is coded by an ordinal via $E^*$, and so $V_{j(\kappa)}=V_\gamma\of N$. Similarly, since $M^{\ltkappa}\of M$, it follows that $\<V_\kappa,{\in},E>$ believes that the structure $\<\kappa,E>$ is closed under $\ltkappa$-sequences (that is, for any $\beta<\kappa$ and any $\beta$-sequence $\<x_\alpha\st\alpha<\beta>$ of ordinals below $\kappa$, there is $s<\kappa$ such that $\<\kappa,E>$ thinks $s$ is a sequence, whose $\alpha^{\rm th}$ member is precisely $x_\alpha$), and so by elementarity the corresponding fact is true of $\<V_\gamma,{\in},E^*>$. Since $V_\gamma$ itself is correct about $[\gamma]^{\ltgamma}$, this implies $N^{\ltgamma}\of N$, or in other words, $N^{<j(\kappa)}\of N$. Finally, since we chose $M$ such that $V_\kappa=V_\kappa^M\elesub M$, it follows by elementarity that $V_{j(\kappa)}=(V_{j(\kappa)})^N\elesub N$, as desired.

($4\implies 5$) The embedding property (4) asserts the existence of $\kappa$-models, which implies $\kappa^{\lt\kappa}=\kappa$, and it then follows that $\kappa$ is inaccessible. If $M$ is a $\kappa$-model, then by statement (4) there is another $\kappa$-model $\bar{M}\satisfies\ZFC$ with $M\in \bar{M}$ and a transitive set $N$ with an embedding $j:\bar{M}\to N$ with critical point $\kappa$ with $j(\kappa)>\theta$ and $V_{j(\kappa)}\of N$, such that $N^{<j(\kappa)}\of N$ and $j(\kappa)$ is weakly compact and more. Since $M^{\ltkappa}\of M$, it follows that $j(M)^{<j(\kappa)}\of j(M)$ inside $N$, and since $N^{<j(\kappa)}\of N$ we know that $N$ is correct about this. It follows that $V_{j(\kappa)}\of N$ as well, and so the restricted embedding $j\restrict M:M\to j(M)$ verifies statement (5).

($5\implies 3$) This direction is immediate, since $\kappa^\ltkappa=\kappa$ implies that every set $A\of \kappa$ can be placed into some $\kappa$-model.

($3\implies 2$) This is immediate, since $N^{<j(\kappa)}\of N$ implies $V_{j(\kappa)}\of N$, as it is easy to see that $j(\kappa)$ must be inaccessible.

($2\implies 1$) Suppose that $\kappa$ is superstrongly $(\theta+1)$-unfoldable. It follows easily that $\kappa$ is inaccessible. To see that $\kappa$ is strongly $(\theta+1)$-uplifting, we verify the extension property of theorem~\ref{T.ExtensionCharacterizations} statement~(1). For any $A\of\kappa$, there is a $\kappa$-model $M$ with $A\in M$ and $j:M\to N$ with critical point $\kappa$, for which $j(\kappa)>\theta$ and $V_{j(\kappa)}\of N$, and consequently $j(V_\kappa)=V_{j(\kappa)}$. If $A^*=j(A)$, then it follows by elementarity that $\<V_\kappa,{\in},A>\elesub\<V_{j(\kappa)},{\in},A^*>$, witnessing this instance of $\kappa$ being strongly $(\theta+1)$-uplifting.
\end{proof}

Note particularly that in the superstrongly unfoldable embedding characterization, there is no stipulation that $j(\kappa)$ must be inaccessible; but nevertheless, by the other embedding characterizations, one may always find alternative superstrong unfoldability embeddings still above $\theta$ for which $j(\kappa)$ is inaccessible, weakly compact and more, just as in theorem~\ref{T.ExtensionCharacterizations}. Theorem~\ref{T.EmbeddingCharacterizations} was stated in terms of the successor ordinal $\theta+1$, a case amounting to the requirement that $j(\kappa)>\theta$, and in this case all the notions are locally equivalent, but similar arguments show that some of the notions are locally equivalent for every ordinal $\theta$, not just successor ordinals. Namely, a cardinal $\kappa$ is strongly $\theta$-uplifting if and only if it is almost-hugely $\theta$-unfoldable. As a result, one should regard almost-hugely $\theta$-unfoldability as the right embedding characterization of strong $\theta$-upliftingness. Meanwhile, when $\theta$ is a singular limit cardinal, these notions are not in general equivalent to $\kappa$ being superstrongly $\theta$-unfoldable, since it can happen that a cardinal $\kappa$ is superstrongly $\theta$-unfoldable for such a singular $\theta$, but not even $\theta$-uplifting, let alone strongly $\theta$-uplifting.

Next, we consider the difference in consistency strength between uplifting cardinals and strongly uplifting cardinals. In the case of unfoldable cardinals, Villaveces~\cite{Villaveces1998:ChainsOfEndElementaryExtensionsOfModelsOfSetTheory} showed that every unfoldable cardinal is unfoldable in $L$, and every unfoldable cardinal in $L$ is strongly unfoldable there. Thus, unfoldability and strong unfoldability are equiconsistent as large cardinal hypotheses. For the case of uplifting and strong uplifting, in contrast, we shall show presently that there is a definite step up in consistency strength. While uplifting cardinals are weaker than Mahlo cardinals in consistency strength, theorems~\ref{T.ConsistLowerBound} and~\ref{T.SubtleUpperBound} show that the consistency strength of the existence of a strongly uplifting cardinal, if consistent, lies strictly between the existence of a strongly unfoldable cardinal and the existence of a subtle cardinal.

In analogy with the various large cardinal Mitchell rank concepts, we defined in~\cite{HamkinsJohnstone2010:IndestructibleStrongUnfoldability} that a strongly unfoldable cardinal $\kappa$ is strongly unfoldable {\df of degree $\alpha$}, for an ordinal $\alpha$, if  for every ordinal $\theta$ it is {\df $\theta$-strongly unfoldable of degree $\alpha$}, meaning that for each $A\in H_{\kappa^\plus}$ there is a $\kappa$-model $M\satisfies \ZFC$ with $A\in M$ and a transitive set $N$ with $\alpha\in N$ and an elementary embedding $j:M\to N$ having critical point $\kappa$ with $j(\kappa)>\max\{\theta,\alpha\}$ and $V_\theta\of N$, such that $\kappa$ is strongly unfoldable of every degree $\beta<\alpha$ in $N$.\footnote{Technically, in~\cite{HamkinsJohnstone2010:IndestructibleStrongUnfoldability} we had only required that the domain $M$ of the elementary embedding $j$ is a transitive set of size $\kappa$ with $M\satisfies\ZFC^{-}$ and $\kappa,A\in M$; however, by restricting such $j$ to a $\kappa$-model $M$ as in footnote~\ref{F.specialkappamodel} with $V_\kappa\elesub M$, if necessary, we may assume without loss that the domain $M$ is a $\kappa$-model satisfying all of $\ZFC$.}
An inaccessible cardinal $\kappa$ is \emph{$\Sigma_2$-reflecting,} if $V_\kappa\elesub_{\Sigma_2}V$, and it is easy to see that every strongly uplifting cardinal and even merely every pseudo-uplifting cardinal is $\Sigma_2$-reflecting.

\begin{theorem} \label{T.ConsistLowerBound}
 If $\kappa$ is strongly uplifting, then $\kappa$ is strongly unfoldable, and furthermore, strongly unfoldable of every ordinal degree $\alpha$, and a stationary limit of cardinals that are strongly unfoldable of every ordinal degree and so on.
\end{theorem}

\begin{proof}
Suppose that $\kappa$ is strongly uplifting, and suppose inductively that $\kappa$ is strongly unfoldable of every ordinal degree $\beta$ below $\alpha$. Since $\kappa$ is strongly unfoldable by theorem~\ref{T.EmbeddingCharacterizations}, we may find (by collapsing a suitable elementary substructure of some large $V_\eta$ when $\eta$ is inaccessible) for any $A\of\kappa$ a $\kappa$-model $M\satisfies\ZFC$ with $A\in M$ such that $M\satisfies\kappa$ is strongly unfoldable, and in particular, such that $\kappa$ is $\Sigma_2$-reflecting in $M$. Since $\kappa$ is strongly uplifting, we may find by theorem~\ref{T.EmbeddingCharacterizations} statement (5) an elementary embedding $j:M\to N$ such that $j(\kappa)$ is inaccessible, $j(\kappa)>\alpha$ and $V_{j(\kappa)}\of N$.

For every ordinal $\theta<j(\kappa)$, we claim that $\kappa$ is $\theta$-strongly unfoldable in $N$ of every degree $\beta<\alpha$. The reason is simply that this holds in $V$ and is witnessed by extender embeddings of size $\max\{\beth_{\theta},\alpha,\kappa\}$, which are therefore inside $V_{j(\kappa)}$ and hence in $N$. Since furthermore $j(\kappa)$ is $\Sigma_2$-reflecting in $N$, this means that any counterexample to strong unfoldability would reflect below
$j(\kappa)$ and so $\kappa$ is fully strongly unfoldable of every ordinal degree $\beta$ below $\alpha$ in $N$. Thus, $\kappa$ is strongly unfoldable in $V$ of degree $\alpha$, and the proof is complete by induction on $\alpha$.

For the second claim, consider any club $C\of\kappa$ and ensure also that $C\in M$ in the argument above. The argument shows that $\kappa$ is $\ltjkappa$-strongly unfoldable of every ordinal degree $\alpha<j(\kappa)$ in $N$, and consequently it is  strongly unfoldable of every ordinal degree in $N$. Since furthermore $\kappa\in j(C)$, this means that $\kappa$ must be a stationary limit of such cardinals in $V$, as there can be no club $C\of\kappa$ containing no such cardinals.
\end{proof}

Thus, the strongly uplifting cardinals subsume the entire hierarchy of degrees of strong unfoldability.

Having now provided a strong lower bound, let us turn to the question of an upper bound. Recall that a cardinal $\kappa$ is \emph{subtle} if for any closed unbounded set $C\of \kappa$ and any sequence $\<A_\alpha\st \alpha\in C>$ with $A_\alpha\of\alpha$, there is a pair of ordinals $\alpha<\beta$ in $C$ with $A_\alpha=A_\beta\intersect \alpha$. It is not difficult to see that every subtle cardinal is necessarily inaccessible. Subtle cardinals need not themselves be unfoldable (see~\cite[Prop 2.4]{Villaveces1998:ChainsOfEndElementaryExtensionsOfModelsOfSetTheory}), and so they need not be strongly uplifting.

\begin{theorem}\label{T.SubtleUpperBound}
 If $\delta$ is a subtle cardinal, then the set of cardinals $\kappa$ below $\delta$ that are strongly uplifting in $V_\delta$ is stationary.\end{theorem}

\begin{proof} The argument is essentially related to~\cite[theorem 2.2]{Villaveces1998:ChainsOfEndElementaryExtensionsOfModelsOfSetTheory} and also~\cite[theorem 3]{DzamonjaHamkins2006:DiamondCanFail}. Suppose that $\delta$ is subtle and the set of cardinals below $\delta$ that are strongly uplifting in $V_\delta$ is not stationary. Then there is a closed unbounded set $C\of\delta$ containing no such cardinals. Since each cardinal in $C$ is not strongly uplifting in $V_\delta$, it follows from statement~(1) of theorem~\ref{T.ExtensionCharacterizations} applied in $V_\delta$ that for each $\kappa\in C$, there is some least $\theta<\delta$ and some subset $A_\kappa \of\kappa$, such that $\<V_\kappa,{\in},A_\kappa>$ has no proper elementary extension of the form $\<V_\gamma,{\in},A^*>$ for any $\gamma$ with $\theta<\gamma<\delta$. By thinning the club $C$, we may assume that $\theta$ is less than the next element of $C$ above $\kappa$, and also that $\kappa$ is a $\beth$-fixed point. Since $V_\kappa$ has size $\kappa$, let $B_\kappa\of\kappa$ be a set that codes the elementary diagram of the structure $\<V_\kappa,{\in},A_\kappa>$ in some uniform canonical manner. Since $\delta$ is subtle, there is a pair $\kappa<\eta$ in $C$ with $B_\kappa=B_\eta\intersect\kappa$. Since $B_\kappa$ and $B_\eta$ code the corresponding elementary diagrams, it follows that those structures agree on their truths below $\kappa$, and so $\<V_\kappa,{\in},A_\kappa>\elesub\<V_\eta,{\in},A_\eta>$. This contradicts the assumption that $\<V_\kappa,{\in},A_\kappa>$ has no proper elementary extension above $\theta$, which is less than the next element of $C$ and therefore less than $\eta$. So the conclusion of the theorem must hold, as desired.
\end{proof}

\begin{theorem}\label{T.ZeroSharp}
 If $0^\sharp$ exists, then every Silver indiscernible is strongly uplifting in $L$.
\end{theorem}

\begin{proof}
If $\kappa$ is any Silver indiscernible, then for any Silver indiscernible $\delta$ above $\kappa$, there is an elementary embedding $j:L\to L$ with $j(\kappa)=\delta$ and $j\restrict\kappa=\id$. If $A\of\kappa$ is any set in $L$, then $\<L_\kappa,{\in},A>\elesub\<L_\delta,{\in},j(A)>$, witnessing the desired instance of strong uplifting.
\end{proof}

So for example, if there is a Ramsey cardinal, then every uncountable cardinal of $V$ is strongly uplifting in $L$.

Let us say that a cardinal $\kappa$ is {\df unfoldable with cardinal targets}, if for every $A\in H_{\kappa^+}$ there is a $\kappa$-model $M$ with $A\in M$, an ordinal $\theta$, a transitive set $N$ and an elementary embedding $j:M\to N$ with critical point $\kappa$, such that $j(\kappa)$ is a cardinal (in $V$) and $j(\kappa)\geq\theta$.

\begin{theorem}
 In the constructible universe $L$, $\kappa$ is strongly uplifting if and only if it is unfoldable with cardinal targets.
\end{theorem}

\begin{proof}
The forward implication holds whether or not we are in $L$, since if $\kappa$ is strongly uplifting, then by theorem~\ref{T.EmbeddingCharacterizations} we get for any $A\in H_{\kappa^+}$ a $\kappa$-model $M$ with $A\in M$ and an embedding $j:M\to N$ with critical point $\kappa$ and $j(\kappa)$ weakly compact and more, as large as desired; and so $\kappa$ is unfoldable with the desired targets. Conversely, assume that $V=L$ and $\kappa$ is unfoldable with cardinal targets. For any $A\of\kappa$ we may find a $\kappa$-model $M$ with $A\in M$, and an embedding $j:M\to N$ with $j(\kappa)$ a cardinal in $V$ and as large as desired. Since $j$ fixes everything of rank below $\kappa$, it follows by elementarity that $\<L_\kappa,{\in},A>\elesub \<L_{j(\kappa)},{\in},j(A)>$. We have $L_\kappa=V_\kappa$ since $\kappa$ is inaccessible. Since $j(\kappa)$ is a cardinal, it follows that $L_{j(\kappa)}$ correctly computes all cardinals and $\beth_\alpha$ below $j(\kappa)$, and so by elementarity $L_\kappa\elesub L_{j(\kappa)}$ it follows that $j(\kappa)$ is a $\beth$-fixed point in $L$. Thus, $L_{j(\kappa)}=V_{j(\kappa)}^L$, thereby witnessing the desired instance of strong upliftingness for $\kappa$, using theorem~\ref{T.ExtensionCharacterizations} statement~(1).
\end{proof}

Let us now turn to the Menas and Laver function concepts for the strongly uplifting cardinals. Define that $f:\kappa\to\kappa$ is a {\df Menas} function for a strongly uplifting cardinal $\kappa$, if for every set $A\of\kappa$ and every $\theta$, there is a proper elementary extension $\<V_\kappa,{\in},A,f>\elesub\<V_\gamma,{\in},A^*,f^*>$, where $\gamma>\theta$ is inaccessible and $f^*(\kappa)\geq\theta$.

\begin{theorem}\label{T.MenasFunction}
 Every strongly uplifting cardinal $\kappa$ has a strongly uplifting Menas function.
\end{theorem}

\begin{proof}
As in~\cite[thm~13]{HamkinsJohnstone2014:ResurrectionAxiomsAndUpliftingCardinals}, we may simply use the failure-of-strong-uplifting function, namely, the function defined by $f(\delta)=\theta$, if $\delta$ is not strongly $\theta$-uplifting, but it is strongly $\beta$-uplifting for every $\beta<\theta$. Suppose that $\kappa$ is strongly uplifting
and consider any ordinal $\theta$ and any $A\of\kappa$. Let $\lambda$ be any ordinal above $\theta$ such that $V_\lambda\satisfies\kappa$ is strongly uplifting\footnote{Note that the various characterizations of strongly uplifting cardinals as in theorems~\ref{T.ExtensionCharacterizations} and~\ref{T.EmbeddingCharacterizations} are all equivalent for such models $V_\lambda$, even though $V_\lambda$ need not satisfy all of $\ZFC$.\label{footnote}}\,, and let $\eta$ be the smallest inaccessible cardinal above $\lambda$ for which there is an elementary extension $\<V_\kappa,{\in},A,f>\elesub\<V_\eta,{\in},A^*,f^*>$. (Note that we needn't actually include $f$ in the language, since it is definable, and $f^*$ will be similarly defined in $V_\eta$.) It follows by the minimality of $\eta$ that $\kappa$ is not strongly uplifting in $V_\eta$, but by the choice of $\lambda$, we know that $V_\eta\satisfies\kappa$ is strongly $\ltlambda$-uplifting. It follows that $f^*(\kappa)\geq\lambda$, which is at least $\theta$, and so we have fulfilled the desired Menas property.
\end{proof}

The Menas function concept interacts well with the embedding characterizations of theorem~\ref{T.EmbeddingCharacterizations}, with the result that one can find embeddings $j:M\to N$ as in that theorem for which $j(f)(\kappa)\geq\theta$.

Although the Menas function concept suffices for many applications, including especially the lottery-style forcing iterations we shall use for the equiconsistency in theorem~\ref{T.Equiconsistency}, nevertheless a more refined analysis results in the Laver function concept. Namely, a function $\ell\from\kappa\to V_\kappa$ is a {\df Laver function} for a strongly uplifting cardinal $\kappa$, if for every $A\of\kappa$, every ordinal $\theta$ and every set $x$, there is a proper elementary extension $\<V_\kappa,{\in},A,\ell>\elesub\<V_\gamma,{\in},A^*,\ell^*>$ where $\gamma\geq\theta$ in inaccessible and $\ell^*(\kappa)=x$.  The function $\ell\from\kappa\to V_\kappa$ is merely an {\df $\OD$-anticipating} Laver function, if this property can be achieved at least for $x\in\OD$, and similarly a function $\ell\from\kappa\to\kappa$ is an {\df ordinal-anticipating} Laver function for a strongly uplifting cardinal $\kappa$, if for every $A\of\kappa$ and any two ordinals $\alpha,\theta$, there is a proper elementary extension $\<V_\kappa,{\in},A,\ell>\elesub\<V_\gamma,{\in},A^*,\ell^*>$ where $\gamma\geq\theta$ is inaccessible and $\ell^*(\kappa)=\alpha$.

\begin{theorem}\label{T.LaverFunctions}
 Every strongly uplifting cardinal $\kappa$ has an ordinal-anticipating Laver function $\ell\from\kappa\to\kappa$, and indeed, an $\OD$-anticipating Laver function $\ell\from\kappa\to V_\kappa$. Furthermore, there is such a Laver function that is definable in $\<V_\kappa,{\in}>$.
\end{theorem}

\begin{proof}
Let us first construct an ordinal-anticipating Laver function. For any cardinal $\delta<\kappa$, consider the set of ordinals $\gamma$ below $\kappa$ for which $V_\gamma\satisfies\delta$ is strongly uplifting. % (see footnote~\ref{footnote})
If this class of ordinals bounded in $\kappa$ and has order type $\xi+1$ for some $\xi$, and if furthermore $\xi=\<\alpha,\beta>$ is the \Godel\ code of a pair of ordinals, then define $\ell(\delta)=\alpha$; otherwise let $\ell(\delta)$ be undefined. Thus, we have defined the function $\ell\from\kappa\to\kappa$.

We claim that this function is an ordinal-anticipating Laver function for the strong upliftingness of $\kappa$. To see this, consider any $A\of\kappa$ and any two ordinals $\alpha,\theta$. Let $\xi=\<\alpha,\theta>$ be the \Godel\ code, which we assume is at least $\theta$ (and we may assume this is at least $\kappa$), and let $\lambda$ be the $(\xi+1)^\th$ ordinal such that $V_\lambda\satisfies\kappa$ is strongly uplifting, and let $\eta$ be the least inaccessible cardinal above $\lambda$ for which there is an extension $\<V_\kappa,{\in},A,\ell>\elesub \<V_\eta,{\in},A^*,\ell^*>$. By the minimality of $\eta$, it follows that $\lambda$ is the largest ordinal below $\eta$ for which $V_\lambda\satisfies\kappa$ is strongly uplifting, and so the set of ordinals $\gamma$ below $\eta$ for which $V_\gamma\satisfies\kappa$ is strongly uplifting has order type
$\xi+1$. This means that $V_\eta\satisfies\ell^*(\kappa)=\alpha$, precisely because it will be using the ordinal $\xi=\<\alpha,\theta>$ when the definition is unraveled. Since $\xi\geq\theta$, we have thereby witnessed the desired instance of the ordinal-anticipating Laver function property.

One may now produce from $\ell$ an $\OD$-anticipating Laver function $\hat\ell\from\kappa\to V_\kappa$, using the same idea as in~\cite[thm~14]{HamkinsJohnstone2014:ResurrectionAxiomsAndUpliftingCardinals}. Let $\hat\ell(\delta)=x$, if $\ell(\delta)=\<\eta,\beta>$ and $x$ is the $\beta^{th}$ ordinal-definable object in $\<V_\eta,{\in}>$. Now, if $x\in\OD$, then $x\in\OD^{V_\eta}$ for some $\eta$, and it is the $\beta^{\rm th}$ ordinal-definable object in $V_\eta$. Since $\ell$ is an ordinal-anticipating Laver function, we may find $\<V_\kappa,{\in},A,\ell>\elesub\<V_\gamma,{\in},A^*,\ell^*>$ for which $\ell^*(\kappa)=\<\eta,\beta>$, and in this case we will have $\hat\ell^*(\kappa)=x$, since $V_\gamma$ will be looking at the $\beta^{\rm th}$ ordinal-definable object of $V_\eta$, which is $x$.
\end{proof}

In particular, if $V=\HOD$, then every strongly uplifting cardinal has a strongly uplifting Laver function. So every strongly uplifting cardinal has a strongly uplifting Laver function in $L$. Following the terminology of~\cite{Hamkins:LaverDiamond}, we say that the Laver diamond principle $\LD_\kappa^{\!\!\!\hbox{\scriptsize\rm str-uplift}}$ holds for a strongly uplifting cardinal $\kappa$, if there is such a Laver function. And so we have proved that $\LD_\kappa^{\!\!\!\hbox{\scriptsize\rm str-uplift}}$ holds under $V=\HOD$ for any strongly uplifting cardinal $\kappa$. Meanwhile, we are unsure whether every strongly uplifting cardinal must have a full Laver function. Perhaps this can fail; it is conceivable that one might generalize the proof of the main theorem of~\cite{DzamonjaHamkins2006:DiamondCanFail} to the superstrong unfoldable context, in order to produce strongly uplifting cardinals lacking $\Diamond_\kappa(\REG)$, which would prevent the existence of Laver functions. We shall leave that question for another project.

\begin{question}
 Is it relatively consistent that a strongly uplifting cardinal has no strongly uplifting Laver function? Can $\Diamond_\kappa(\REG)$ fail when $\kappa$ is strongly uplifting?
\end{question}

Finally, let us remark that the Laver functions interact well with the embedding characterizations of theorem~\ref{T.EmbeddingCharacterizations}, with the effect that after tracing through the equivalences, one finds the corresponding embeddings $j:M\to N$, for which $j(\ell)(\kappa)$ has the desired value.

%%%%%%%%%%%%%%%%%%%%%%%%%%%%%%%%%%%%%%%%%%%%%%%%%%%%%%%%%%%%%%%%%%%%%%%%%%%%%%%%%%%%%%%%%%%%%%%%%%%%%%%%%%%%%%%%%%%%%%%%%%%%%%%%%%%%%%%%%%%%%%%%%%%%%%%%%%%%%%%%%%%%%%%%%%%%%%%%%%%%%%%%%%%%%%%%%%%%%%%%%%%%%%%%%%%%%
%%%%%%%%%%%%%%%%%%%%   section:      BOLDFACE RESURRECTION    %%%%%%%%%%%%%%%%%%%%%%%%%%%%%%%%%%%%%%
\section{The boldface resurrection axioms}\label{S.BoldfaceRA}

We shall aim in section~\ref{S.Equiconsistency} to prove that the existence of a strongly uplifting cardinal is equiconsistent with the boldface resurrection axioms, which we shall now introduce. These axioms generalize and strengthen many instances of the bounded forcing axioms that are currently a focus of investigation in the set-theoretic research community. The main idea is simply to generalize the resurrection axioms of~\cite{HamkinsJohnstone2014:ResurrectionAxiomsAndUpliftingCardinals} by allowing an arbitrary parameter $A$. We use the notation $\continuum$ to denote the continuum, that is, the cardinality $\continuum=2^\omega=|\R|$, and $\Hc$ denotes the collection of sets of hereditary size less than $\continuum$.

\begin{definition}\rm
 Suppose that $\Gamma$ is any definable class of forcing notions.
\begin{enumerate}
\item The boldface resurrection axiom $\RAbf(\Gamma)$ asserts that for every $\Q\in\Gamma$ and $A\of \continuum$, there is
      $\Rdot\in\Gamma^{V^\Q}$ such that if $g*h\of\Q*\Rdot$ is $V$-generic, then there is $A^*$ in $V[g*h]$ with
      $$\<\Hc,{\in},A>\elesub\<\Hc^{V[g*h]},{\in},A^*>.$$
 \item The weak boldface resurrection axiom $\wRAbf(\Gamma)$ drops the requirement that $\Rdot$ needs to be in $\Gamma^{V[g]}$.
\end{enumerate}
\end{definition}

These boldface resurrection axioms naturally strengthen the corresponding lightface versions $\RA(\Gamma)$ and $\wRA(\Gamma)$, which were the main focus of~\cite{HamkinsJohnstone2014:ResurrectionAxiomsAndUpliftingCardinals}, as the lightface forms amount simply to the special case of where $A$ is trivial or simply omitted. One may easily observe that $\RAbfall$ implies $\wRAbf(\Gamma)$ for any class $\Gamma$ of forcing notions, and $\RAbf(\Gamma)$ implies $\wRAbf(\Gamma)$. Moreover, if $\Gamma_1\of\Gamma_2$ are two classes of forcing notions, then $\wRAbf(\Gamma_2)$ implies $\wRAbf(\Gamma_1)$, but in general $\RAbf(\Gamma_2)$ need not imply $\RAbf(\Gamma_1)$.\footnote{For example, theorem~\ref{T.Equiconsistency} shows that $\RAbfall$ is consistent relative to a strongly uplifting cardinal, whilst corollary~\ref{Corollary.RAbf(Aleph1-preserving)Inconsistent} shows $\RAbf(\aleph_1\text{-preserving})$ is inconsistent, even though $\aleph_1\text{-preserving}\of \text{all}$.} Many further observations about the resurrection axioms made in~\cite{HamkinsJohnstone2014:ResurrectionAxiomsAndUpliftingCardinals} relativize easily to the boldface case.

If $\Gamma$ is a class of forcing notions and $\kappa$ and $\delta$ are cardinals, then the \emph{bounded forcing axiom $\BFA_\delta(\Gamma,\kappa)$}, introduced by Goldstern and Shelah~\cite{GoldsternShelah1995Nr507:BPFA}, is the assertion that whenever $\Q\in\Gamma$ and $\B=\text{r.o.}(\Q)$, if $\mathcal{A}$ is a collection of at most $\kappa$ many maximal antichains in $\B\setminus\{0\}$, each antichain of size at most $\delta$, then there is a filter on $\B$ meeting each antichain in $\mathcal{A}$. This axiom therefore places limitations both on the number of antichains to be considered, as well as on the sizes of those antichains. To simplify notation, the bounded forcing axiom $\BFA_\kappa(\Gamma,\kappa)$ is denoted more simply as $\BFA_\kappa(\Gamma)$; the $\delta$-bounded proper forcing axiom $\BFA_\delta(\proper,\omega_1)$ is denoted as $\PFA_\delta$; the $\delta$-bounded semi-proper forcing axiom $\BFA_\delta(\semiproper,\omega_1)$ is denoted $\SPFA_\delta$; the analogous $\delta$-bounded forcing axiom for axiom-A posets is denoted $\rm AAFA_\delta$; and the $\delta$-bounded forcing axiom for forcing that preserves stationary subsets of $\omega_1$ is denoted $\MM_\delta$.

\begin{theorem}\label{T.RAbfImplyBoundedFA}
 $\wRAbf(\Gamma)$ implies $\BFA_\continuum(\Gamma,\kappa)$ for any cardinal $\kappa<\continuum$.
\end{theorem}

\begin{proof}
This argument extends the argument of~\cite[thm~4]{HamkinsJohnstone2014:ResurrectionAxiomsAndUpliftingCardinals} to antichains of size $\continuum$, the point being that the boldface hypothesis allows us to handle such larger antichains by treating them as predicates on $\Hc$ rather than as elements of $\Hc$. Assume $\wRAbf(\Gamma)$ and fix any cardinal $\kappa<\continuum$. Suppose that $\B=\text{r.o.}(\Q)$ for some poset $\Q\in\Gamma$, and ${\mathcal A}=\{A_\alpha\st\alpha<\kappa\}$ is a collection of at most $\kappa$ many maximal antichains in $\B\setminus\{0\}$, each antichain of size at most $\continuum$. Let $\B_0$ be the subalgebra of $\B$ generated by $\Union_\alpha A_\alpha$. This has size at most $\continuum$, and so by replacing $\B$ with an isomorphic copy we may assume that both ${\mathcal A}$ and $\B_0$ are subsets of $\Hc$ of size $\continuum$. Let $g\of\B$ be a $V$-generic filter. It follows that $g$ is also $\mathcal A$-generic, and so $g\intersect\B_0$ meets every $A_\alpha$. By the $\wRAbf(\Gamma)$, there is some further forcing $h\of\dot\R$ after which we may find an elementary extension  $\<\Hc,{\in},\B_0,A_\alpha>_{\alpha<\kappa}\elesub\<\Hc^{V[g][h]},{\in},\B_0^*,A_\alpha^*>_{\alpha<\kappa}$, using the fact that we may code all this additional structure into a single predicate on $\continuum$. For each $\alpha<\kappa$, let $p_\alpha\in g\intersect \B_0\intersect A_\alpha$, and let $F=\set{p_\alpha\st\alpha<\kappa}$, which is a set of size $\kappa$ in $V[g]$ and hence in $\Hc^{V[g][h]}$, which has the finite-intersection property and meets every antichain $A_\alpha$. By elementarity, therefore, there must be a such a set already in $\Hc^V$, and the filter in $\B$ generated by this set will meet every $A_\alpha$, thereby witnessing the desired instance of $\BFA_\continuum(\Gamma,\kappa)$.
\end{proof}

We have the following immediate corollary.

\begin{corollary} \label{C:ConseqwRAbf}\
\begin{enumerate}
 \item $\wRAbf(\proper)+\neg\CH$ implies $\PFA_\continuum$.
 \item $\wRAbf(\semiproper)+\neg\CH$ implies $\SPFA_\continuum$.
 \item $\wRAbf(\text{axiom-A})+\neg\CH$ implies $\rm AAFA_\continuum$.
 \item $\wRAbf(\text{preserving stationary subsets of }\omega_1)+\neg\CH$ implies $\MM_\continuum$
\end{enumerate}
\end{corollary}

The conclusion $\PFA_\continuum$ of (1) is equiconsistent with the existence of an $H_{\kappa^+}$-reflecting cardinal, by a result due to Miyamoto~\cite{Miyamoto1998:WeakSegmentsOfPFA}, and $H_{\kappa^+}$-reflecting cardinals are exactly the same as strongly unfoldable cardinals. The same is true for the conclusion $\SPFA_\continuum$ of (2). Miyamoto's argument~\cite{Miyamoto1998:WeakSegmentsOfPFA} shows in fact that $\rm AAFA_{\omega_2}$ is sufficient to conclude that $\omega_2$ is strongly unfoldable in $L$ and so the conclusion $\rm AAFA_\continuum$ of (3) is also equiconsistent with the existence of a strongly unfoldable cardinal. The failure of $\CH$ is of course a necessary hypothesis in statements (1)-(4) of Corollary~\ref{C:ConseqwRAbf}, because the conclusions imply $\neg\CH$, while all the weak boldface resurrection axioms are compatible with \CH, in light of the fact that they are all implied by $\RAbfall$, which implies \CH\ by~\cite[thm~5]{HamkinsJohnstone2014:ResurrectionAxiomsAndUpliftingCardinals}.

The boldface resurrection axioms admit the following useful embedding characterization, which is analogous to that of the strongly uplifting cardinals. The hypothesis $|\Hc|=\continuum$, a consequence of \MA, will hold in the principal cases in which we shall be interested.

\begin{theorem}[Embedding characterization of boldface resurrection]
If $|\Hc|=\continuum$, then the following are equivalent for any class $\Gamma$.\label{T.RAbfEmbedChar}
\begin{enumerate}
 \item The boldface resurrection axiom $\RAbf(\Gamma)$.
\item For every $\Q\in\Gamma$ and every transitive set $M\satisfies\ZFC^-$ with $|M|=\continuum\in M$ and $\Hc\of M$, there is
       $\Rdot\in\Gamma^{V^\Q}$, such that in any forcing extension $V[g*h]$ by $\Q*\Rdot$, there is an elementary
      embedding $$j:M\to N$$ with $N$ transitive, $j\restrict\Hc=\id$, $j(\continuum)=\continuum^{V[g*h]}$ and $\Hc^{V[g*h]}\of N$.
\end{enumerate}
Similarly, the weak boldface axiom $\wRAbf(\Gamma)$ is equivalent to the embedding characterization obtained by omitting the requirement that $\Rdot\in\Gamma^{V^\Q}$.
\end{theorem}

\begin{proof} Suppose that $\RAbf(\Gamma)$ holds. Fix any $\Q\in\Gamma$ and any transitive $M\satisfies\ZFC^-$ with $|M|=\continuum\in M$ and $\Hc\of M$. Find a relation $E$ on $\continuum$ and an isomorphism $\tau:\<M,{\in}>\cong\<\continuum,E>$, and let $A\of\continuum$ code $E$ via a canonical
pairing function. By $\RAbf(\Gamma)$, there is $\Rdot\in\Gamma^{V^\Q}$ such that if $g*h\of\Q*\Rdot$ is $V$-generic, then in $V[g*h]$ there is a set
$A^*\of \continuum^{V[g*h]}$ such that $\<\Hc,{\in},A>\elesub\<\Hc^{V[g*h]},{\in},A^*>$. The set $A^*$ codes a relation $E^*$ on $\continuum^{V[g*h]}$. Since
$\<\Hc,{\in},A>$ knows that $E$ is well founded, it follows by elementarity that $\<\Hc^{V[g*h]},{\in},A^*>$ thinks $E^*$ is well founded. Since
$\continuum^{V[g*h]}$ has uncountable cofinality, this structure is closed under countable sequences in $V[g*h]$, and $E^*$ really is well founded in
$V[g*h]$. Let $\pi:\<\continuum,E^*>\cong\<N,{\in}>$ be the Mostowski collapse, and let $j=\pi\compose\tau:M\to N$, which can be considered as the three-step composition
$$\<M,{\in}>\quad \cong_\tau\quad \<\continuum,E>\quad \elesub\quad \<\continuum^{V[g*h]},E^*>\quad \cong_\pi\quad \<N,{\in}>,$$ which is therefore
elementary. Note that $j$ is the identity on objects in $\Hc$, since if $x=\tau^{-1}(\alpha)\in\Hc$, then $\<\Hc,{\in},A>$ knows that $x$ is represented by
$\alpha$ with respect to $E$, and so $\<\Hc^{V[g*h]},{\in},A^*>$ agrees that $x$ is represented by $\alpha$ with respect to $E^*$. Similarly, if
$\continuum$ is represented by $\beta$ with respect to $E$, then this can be verified in $\<\Hc,{\in},A>$, and so $\continuum^{V[g*h]}$ is represented by
$\beta$ with respect to $E^*$. Thus, $j(\continuum)=\continuum^{V[g*h]}$, as desired. Finally, $\<\Hc,{\in},A>$ knows that every object in $\Hc$ is
represented in $\<\continuum,E>$, so $\<\Hc^{V[g*h]},{\in},A^*>$ can verify that $\Hc^{V[g*h]}\of N$.

Conversely, suppose that (2) holds. Fix any $\Q\in\Gamma$ and any $A\of\continuum$. Since $|\Hc|=\continuum$, we may find a transitive $M\elesub
H_{\continuum^\plus}$ with $A,\continuum\in M$ and $\Hc\of M$ and $|M|=\continuum$. By (2), there is $\Rdot\in\Gamma^{V^\Q}$ such that in the corresponding forcing
extension $V[g*h]$ we have an embedding $j:M\to N$ with $N$ transitive, $j\restrict\Hc=\id$ and $j(\continuum)=\continuum^{V[g*h]}$ and $\Hc^{V[g*h]}\of N$. Restricting $j$
to $\<\Hc,{\in},A>$, we see that $\<\Hc,{\in},A>\elesub \<\Hc^{V[g*h]},{\in},j(A)>$, verifying this instance of the boldface resurrection axiom $\RAbf(\Gamma)$.

The same argument works in the case of the weak boldface axiom $\wRAbf(\Gamma)$, by omitting the requirement that $\Rdot\in\Gamma^{V^\Q}$.
\end{proof}

While~\cite[thm~6]{HamkinsJohnstone2014:ResurrectionAxiomsAndUpliftingCardinals} shows in the lightface context that the resurrection axioms $\RA(\text{proper})$ and $\RA(\text{semi-proper})$ and others are relatively consistent with $\CH$, this is no longer true in the boldface context.

\begin{theorem}
 If some $\Q$ in $\Gamma$ adds a real and forcing in $\Gamma$ necessarily preserves $\aleph_1$, then the boldface resurrection axiom $\RAbf(\Gamma)$ implies $\neg\CH$. Consequently, the boldface resurrection axioms for proper forcing, semi-proper forcing, and forcing that preserves stationary subsets of $\aleph_1$, respectively, each imply that the continuum is $\continuum=\omega_2$.\label{T.RAbfImplyNegCH}
\end{theorem}

\begin{proof}
Suppose that $\RAbf(\Gamma)$ and \CH\ hold. Let $A\of\omega_1=\continuum$ code all the reals of $V$, and let $\Q\in\Gamma$ be a forcing
notion adding at least one new real. By $\RAbf(\Gamma)$, there is a forcing notion $\Rdot$ in $\Gamma^{V^\Q}$ such that if $g*h\of\Q*\Rdot$ is
$V$-generic, then there is $A^*$ in $V[g*h]$ with $\<\Hc,{\in},A>\elesub\<\Hc^{V[g*h]},{\in},A^*>$. Since \CH\ holds, $\Hc=H_{\omega_1}$ and this structure
believes that every object is countable. By elementarity this is also true in $\Hc^{V[g*h]}$ and so the \CH\ holds in $V[g*h]$. Since we assumed that
forcing in $\Gamma$ preserves $\aleph_1$, it follows that $\omega_1^{V[g*h]}=\omega_1$. From this, it follows that $A^*=A$, and so the new real does
not appear on $A^*$, a contradiction.

It follows that the boldface resurrection axioms in the case of proper forcing, semi-proper forcing, and so on each imply $\continuum=\omega_2$,
since the argument just given shows they imply $\continuum$ is at least $\omega_2$, and it is at most $\omega_2$ by~\cite[thm~ 5]{HamkinsJohnstone2014:ResurrectionAxiomsAndUpliftingCardinals}.\end{proof}

Of course, to make the $\continuum=\omega_2$ conclusion, we didn't use the full power of the boldface axioms, with arbitrary predicates, but only
a single predicate $A$ as above, a well-order of the reals.

Essentially the same arguments as in~\cite[thm~8]{HamkinsJohnstone2014:ResurrectionAxiomsAndUpliftingCardinals} show that several instances of boldface resurrection axioms, including $\RAbf(\text{countably closed})$, $\RAbf(\text{countably distributive})$, $\RAbf(\text{does not add reals})$, and also the weak forms  $\wRAbf(\text{does not add reals})$ and $\wRAbf(\text{countably distributive})$, are each equivalent to the continuum hypothesis~$\CH$. And as in the case of the lightface resurrection axioms, it follows more generally, for any regular uncountable cardinal $\delta$, that each of the  boldface resurrection axioms $\RAbf(\ltdelta\text{-closed})$, $\RAbf(\ltdelta\text{-distributive})$, and $\RAbf(\text{does not add bounded}\text{ subsets of }\delta)$ is equivalent to the assertion that $\continuum\leq\delta$.

Also, the inconsistencies mentioned in~\cite[thm~9]{HamkinsJohnstone2014:ResurrectionAxiomsAndUpliftingCardinals} extend to the boldface context. In addition, we have the following.

\begin{corollary}\label{Corollary.RAbf(Aleph1-preserving)Inconsistent}
The boldface resurrection axiom $\RAbf(\aleph_1\text{-preserving})$ is inconsistent.
\end{corollary}

\begin{proof} On the one hand, the axiom $\RAbf(\aleph_1\text{-preserving})$ implies $\neg\CH$ by theorem~\ref{T.RAbfImplyNegCH}. On the other hand, since $\aleph_1$-preserving forcing can destroy a stationary subset of $\omega_1$, we have that $\RAbf(\aleph_1\text{-preserving})$ implies \CH\ by the remarks after~\cite[thm~5]{HamkinsJohnstone2014:ResurrectionAxiomsAndUpliftingCardinals}.\end{proof}

\section{Strongly uplifting cardinals are equiconsistent with boldface resurrection}\label{S.Equiconsistency}

Let us now prove that the existence of a strongly uplifting cardinal is equiconsistent with various natural instances of the boldface resurrection axiom.

\begin{theorem}\label{T.Equiconsistency}
The following theories are equiconsistent over \ZFC.
\begin{enumerate}
 \item There is a strongly uplifting cardinal.
 \item There is a superstrongly unfoldable cardinal.
 \item There is an almost-hugely unfoldable cardinal.
 \item The boldface resurrection axiom for all forcing.
 \item The boldface resurrection axiom for proper forcing.
 \item The boldface resurrection axiom for semi-proper forcing.
 \item The boldface resurrection axiom for c.c.c.~forcing.
 \item The weak boldface resurrection axioms for countably-closed forcing, for axiom-A forcing, for proper forcing and for semi-proper forcing, respectively, plus $\neg\CH$.
\end{enumerate}
\end{theorem}

\begin{proof}
On the one hand, we shall show that each of these boldface resurrection axioms implies that the continuum $\continuum$ is strongly uplifting in $L$; and conversely, if there is a strongly uplifting cardinal $\kappa$, then we'll explain how to achieve the various boldface resurrection axioms in suitable forcing extensions. Meanwhile, the large cardinal properties of (1), (2) and (3) are equivalent by theorem~\ref{T.EmbeddingCharacterizations}, and hence also equiconsistent.

To begin, suppose that the boldface resurrection axiom~$\RAbfall$ holds. This implies \CH\ by~\cite[thm~5]{HamkinsJohnstone2014:ResurrectionAxiomsAndUpliftingCardinals}. We claim that $\kappa=\continuum=\omega_1$ is strongly uplifting in $L$. Fix any $A\of\kappa$ in $L$, and choose any large ordinal $\theta$. Let $\Q=\Coll(\omega,\theta)$ be the forcing to collapse $\theta$ to $\omega$. Since $A\in L_\alpha$ for some
$\alpha<\kappa^{\plus L}$, there is a transitive set $M\satisfies\ZFC^-$ having $|M|=\continuum\in M$ and $\Hc\of M$ with $A\in L^M$. By \CH\ we know $|\Hc|=\continuum$, and so by theorem~\ref{T.RAbfEmbedChar}, there is $\Rdot$ such that
if $g*h\of\Q*\Rdot$ is $V$-generic, then in $V[g*h]$ there is a transitive set $N$ and an embedding $j:M\to N$ with critical point $\kappa$, such
that $j(\kappa)=\continuum^{V[g*h]}$ and $\Hc^{V[g*h]}\of N$. It follows that $\<L_\kappa,{\in},A>\elesub \<L_{j(\kappa)},{\in},j(A)>$, and since $j(A)$ is in $L^N$, this entire extension is in $L$. Since $j(\kappa)$ is a cardinal in $V[g*h]$, it is a cardinal in $L$, and by the elementarity of $L_\kappa\elesub L_{j(\kappa)}$, it will be a limit cardinal and thus a strong limit cardinal in $L$. By~\cite[thm~4]{HamkinsJohnstone2014:ResurrectionAxiomsAndUpliftingCardinals}, we know that $\MA$ holds in $V$ and this is verified in $\Hc$ and hence in $M$. Thus, $\MA$ holds in $N$
and hence in $\Hc^{V[g*h]}$ and hence in $V[g*h]$. So $\continuum^{V[g*h]}=j(\kappa)$ is regular in $V[g*h]$ and hence regular in $L$. It follows
that $j(\kappa)$ is inaccessible in $L$, and so we have $L_\kappa=(V_\kappa)^L$ and $L_{j(\kappa)}=(V_{j(\kappa)})^L$. Since $\theta$ could have been
made arbitrarily large, we have established that $\kappa$ is strongly uplifting in $L$. So the consistency of (4) implies that of (1).

We can argue similarly if either (5) or (6) holds. In this case, we use instead the poset $\Q=\Coll(\omega_1,\theta)$, but otherwise argue similarly that $\kappa=\continuum=\omega_2$ is strongly uplifting in $L$. If $A\of\kappa$ is in $L$, there is $M$ as above with $A\in L^M$. By (5) or (6) we find $\Rdot$ such that if $g*h\of\Q*\Rdot$ is $V$-generic, then in $V[g*h]$ there is $j:M\to N$ with critical point $\kappa$ and having
$j(\kappa)=\continuum^{V[g*h]}$. Once again, by $\MA$ considerations, $j(\kappa)$ is regular in $V[g*h]$ and hence in $L$, and $\<L_\kappa,{\in},A>\elesub \<L_{j(\kappa)},{\in},j(A)>$. Since $\kappa$ and $j(\kappa)$ are inaccessible in $L$, again we have $L_\kappa=(V_\kappa)^L$ and $L_{j(\kappa)}=(V_{j(\kappa)})^L$. Since $\neg\CH$ holds in $V$ and this is transferred to $V[g*h]$, we know that $j(\kappa)=\continuum^{V[g*h]}$ is
larger than $\theta$. So $\kappa$ is strongly uplifting in $L$. Thus, the consistency of either (5) or (6) implies that of (1).

In the case of (7), we have the boldface resurrection axiom for c.c.c.~forcing $\RAbf(\text{c.c.c.})$, and we use essentially the same argument, but with $\Q=\Add(\omega,\theta)$, where we add $\theta$ many Cohen reals. Note that by~\cite[thm 7]{HamkinsJohnstone2014:ResurrectionAxiomsAndUpliftingCardinals}, the continuum $\continuum$ is a weakly inaccessible cardinal, which is therefore inaccessible in $L$. If $A\of\kappa$ in $L$, there is $M$ as above with $A\in L^M$, and by $\RAbf(\text{c.c.c.})$ there is further c.c.c.~forcing $\dot \R$, such that if $g*h\of\Q*\dot\R$ is $V$-generic, then in $V[g][h]$ there is $j:M\to N$ with critical point $\kappa=\continuum$ having $j(\kappa)=\continuum^{V[g][h]}$, and so $\<L_\kappa,{\in},A>\elesub\<L_{j(\kappa)},{\in},j(A)>$ is the desired extension, which is in $L$ because $j(A)\in L^N$.

Each of the axioms mentioned in (8) implies $\wRAbf(\text{countably-closed})$, and this axiom can be treated the same as $\RAbf(\text{proper})$, since it implies the continuum is at most $\omega_2$ by~\cite[thm~5]{HamkinsJohnstone2014:ResurrectionAxiomsAndUpliftingCardinals} and hence equal to $\omega_2$ by our $\neg\CH$ assumption, and so we may take $\Q=\Coll(\omega_1,\theta)$ and proceed as in the case of (5) above, concluding that $\continuum$ is strongly uplifting in $L$.

Conversely, suppose that $\kappa$ is strongly uplifting. We shall produce the various boldface resurrection axioms in various suitable forcing extensions. For $\RAbfall$, let $\P=\Coll(\omega,\ltkappa)$ be the \Levy\ collapse of $\kappa$, and suppose that $G\of\P$ is $V$-generic. We argue as in~\cite[thms 18,19]{HamkinsJohnstone2014:ResurrectionAxiomsAndUpliftingCardinals}, but with parameters. Fix any $A\of\kappa$ in $V[G]$. There is a name $\Adot\in V$ such that $A=\Adot_G$, and we may assume $\Adot$ has hereditary size $\kappa$ and is coded by a set $A'\of\kappa$ in $V$. Now consider any poset $\Q\in V[G]$. Since $\kappa$ is strongly uplifting, there is a large inaccessible cardinal $\gamma$, above $\kappa$ and $|\dot{\Q}|$, such that $\<V_\kappa,{\in},A'>\elesub \<V_\gamma,{\in},A^*>$ for some $A^*\of\gamma$. Let $\P^*=\Coll(\omega,\ltgamma)$ be the \Levy\ collapse of $\gamma$. The forcing $\P*\Qdot$ is absorbed by this larger collapse, and so there is some quotient forcing $\Rdot$ such that $\P*\Q*\Rdot$ is forcing equivalent to $\P^*$. We may perform further forcing $g*h\of\Q*\Rdot$ and rearrange this to $G^*\of\P^*$, agreeing with $G$ on $\P$, such that $V[G][g*h]=V[G^*]$. By~\cite[lemma 17]{HamkinsJohnstone2014:ResurrectionAxiomsAndUpliftingCardinals}, we may lift the elementary extension to $\<V_\kappa[G],{\in},A',G>\elesub\<V_\gamma[G^*],{\in},A^*,G^*>$. Since $A'$ codes $\Adot$, this implies $\<V_\kappa[G],{\in},A,G>\elesub\<V_\gamma[G^*],{\in},B,G^*>$, where $B$ is the value of the name coded by $A^*$ using $G^*$. Since $\kappa$ and $\gamma$ are inaccessible in $V$, it follows after the \Levy\ collapse that $V_\kappa[G]=\Hc^{V[G]}$ and $V_\gamma[G^*]=\Hc^{V[G^*]}$, so this extension witnesses the desired instance of $\RAbfall$ in $V[G]$. So the consistency of (1) implies that of (4).

Similarly, from (1) we now force (5) using the \PFA\ lottery preparation, introduced in~\cite{Johnstone2007:Dissertation} (used independently in~\cite{NeemanSchimmerling2008:HierarchiesOfForcingAxioms}) based on the lottery iteration idea of~\cite{Hamkins2000:LotteryPreparation}. Suppose that $\kappa$ is strongly uplifting and that $G\of\P$ is $V$-generic for the \PFA\ lottery preparation $\P$, defined with respect to the Menas function $f$ constructed in theorem~\ref{T.MenasFunction}. We show that $V[G]\satisfies\RAbf(\proper)$. We know $\kappa=\continuum^{V[G]}=\aleph_2^{V[G]}$, and $V[G]$ satisfies the lightface $\RA(\proper)$. Suppose that $\Q$ is any proper forcing in $V[G]$ and $A\of\kappa$. Since $\P$ is $\kappa$-c.c., there is a name $\Adot\in H_{\kappa^+}^V$ with $A=\Adot_G$. By the Menas property, and by coding this extra structure into a subset of $\kappa$, we may find an inaccessible cardinal $\gamma$ and an extension $\<V_\kappa,{\in},\P,\Adot,f>\elesub\<V_\gamma,{\in},\P^*,\Adot^*,f^*>$, such that $f^*(\kappa)$ is as large as desired, and in particular above $|\Qdot|$. Since $\P^*$ is the \PFA\ lottery preparation of length $\gamma$ as defined in $V_\gamma$ from $f^*$, it follows that $\Q$ appears in the stage $\kappa$ lottery, since $\Q$ is proper in $V_\gamma[G]$. Below a condition opting for $\Q$ at stage $\kappa$, we may therefore factor $\P^*$ as $\P*\Q*\Ptail$. Suppose that $g*\Gtail\of\Q*\Ptail$ is $V[G]$-generic. By~\cite[lemma~17]{HamkinsJohnstone2014:ResurrectionAxiomsAndUpliftingCardinals}, we may lift the elementary extension to $\<V_\kappa[G],{\in},\P,\Adot,f,G>\elesub\<V_\gamma[G^*],{\in},\P^*,\Adot^*,f^*,G^*>$, where $G^*=G*g*\Gtail$. Since $A$ is definable from $\Adot$ and $G$, it follows that $\<V_\kappa[G],{\in},A>\elesub\<V_\gamma[G^*],{\in},A^*>$, where $A^*=\Adot^*_{G^*}$. Since $\P^*$ is the countable support iteration of proper forcing, it follows that $\Rdot=\Ptail$ is proper in $V[G][g]$, and since $V_\kappa[G]=\Hc^{V[G]}$ and $V_\gamma[G^*]=\Hc^{V[G][g][\Gtail]}$, we have established the desired instance of $\RAbf(\proper)$ in $V[G]$. Thus, the consistency of (1) implies that of (5).

A similar argument, using revised countable support and semi-proper forcing via the \SPFA\ lottery preparation, shows that the consistency of (1)
implies that of (6) as well.

Next, we explain how to force $\RAbf(\text{c.c.c.})$ from a strongly uplifting cardinal $\kappa$. As we discuss in connection with~\cite[thm~20]{HamkinsJohnstone2014:ResurrectionAxiomsAndUpliftingCardinals}, where we consider the lightface resurrection axiom for c.c.c.~forcing, the lottery-style iterations do not work with c.c.c.~forcing, since an uncountable lottery sum of c.c.c.~forcing is no longer c.c.c.  But nevertheless, as in the lightface context, one may proceed with the Laver/Baumgartner-style iteration using the Laver function. By theorem~\ref{T.LaverFunctions}, we may assume without loss that there is a Laver function $\ell\from\kappa\to V_\kappa$ for the strongly uplifting cardinal $\kappa$. Let $\P$ be the finite support c.c.c.~iteration of length $\kappa$, which forces at stage $\beta$ with $\ell(\beta)$, provided that this is a $\P_\beta$-name for c.c.c.~forcing. Suppose that $G\of\P$ is $V$-generic, and consider $V[G]$, where we claim that $\RAbf(\text{c.c.c.})$ holds. To see this, note first that $\kappa=\continuum^{V[G]}$, because unboundedly often the Laver function will instruct us to add another Cohen real. Suppose that $A\of\kappa$ in $V[G]$, that $\Q$ is c.c.c.~forcing there and consider any ordinal $\theta$. Let $\Adot,\Qdot$ be $\P$-names for $A$ and $\Q$, respectively.  Since $\kappa$ is strongly uplifting and $\ell$ is a Laver function, there is an extension $\<V_\kappa,{\in},\Adot,\P,\ell>\elesub\<V_\gamma,{\in},\Adot^*,\P^*,\ell^*>$, with $\gamma\geq\theta$ inaccessible and $\ell^*(\kappa)=\Qdot$. Note that $\P$ is definable from $\ell$, so we needn't have included it in the structure, but $\P^*$ is the corresponding $\gamma$-iteration defined from $\ell^*$, and furthermore $\P^*=\P*\Qdot*\Rdot$, where $\Rdot$ is the rest of the iteration after stage $\kappa$ up to $\gamma$, which is c.c.c.~in $V[G][g]$ since it is a finite-support iteration of c.c.c.~forcing. Let $g*h\of\Q*\Rdot$ be $V[G]$-generic, and by~\cite[lemma~17]{HamkinsJohnstone2014:ResurrectionAxiomsAndUpliftingCardinals} we may lift the extension to $\<V_\kappa[G],{\in},\Adot,\P,\ell,G>\elesub\<V_\gamma[G^*],{\in},\Adot^*,\P^*,\ell^*,G^*>$, where $G^*=G*g*h$. Since $A$ is definable from $\Adot$ and $G$, it follows that $\<V[G]_\kappa,{\in},A>\elesub\<V_\gamma^{V[G^*]},{\in},A^*>$, where $A^*=\Adot^*_{G^*}$. Since $V[G]_\kappa=H_\kappa^{V[G]}$ and $V_\gamma^{V[G^*]}=H_\gamma^{V[G][g][h]}$, and furthermore $\gamma=\continuum^{V[G][g][h]}$, this witnesses the desired instance of $\RAbf(\text{c.c.c.})$ in $V[G]$.

Finally, to achieve models of the axioms mentioned in statement (8) from a strongly uplifting cardinal, note that each of them is implied by $\RAbf(\text{semi-proper})$, and we've already achieved that in statement (6).
\end{proof}

Theorem~\ref{T.Equiconsistency} therefore illustrates how the equiconsistency established in~\cite{HamkinsJohnstone2014:ResurrectionAxiomsAndUpliftingCardinals} between the uplifting cardinals and the resurrection axioms generalizes to the boldface context, with the strongly uplifting cardinals and the boldface resurrection axioms.

\bibliographystyle{alpha}
\bibliography{MathBiblio,HamkinsBiblio}

\end{document}

%% file: MathMacrosJDH.tex
\newtheorem{theorem}{Theorem}

\newtheorem{corollary}[theorem]{Corollary}

\newtheorem{question}[theorem]{Question}

\newtheorem{definition}[theorem]{Definition}

\newcommand{\QED}{\end{proof}}

\def\proclaim[#1]{{\bf #1}}
\def\BF#1.{{\bf #1.}}
\newcommand{\Los}{\L o\'s}

\newcommand{\Godel}{G\"odel}

\newcommand{\Levy}{L\'{e}vy}

\newcommand{\B}{{\mathbb B}}

\renewcommand{\P}{{\mathbb P}}
\newcommand{\Q}{{\mathbb Q}}

\newcommand{\R}{{\mathbb R}}

\newcommand{\continuum}{\mathfrak{c}}

\newcommand{\Gtail}{{G_{\!\scriptscriptstyle\rm tail}}}

\newcommand{\Ptail}{{\P_{\!\scriptscriptstyle\rm tail}}}

\newcommand{\Adot}{{\dot A}}

\newcommand{\Qdot}{{\dot\Q}}

\newcommand{\Rdot}{{\dot\R}}

\newcommand{\id}{\mathop{\hbox{\small id}}}
\newcommand{\from}{\mathbin{\vbox{\baselineskip=2pt\lineskiplimit=0pt
                         \hbox{.}\hbox{.}\hbox{.}}}}

\newcommand{\of}{\subseteq}

\newcommand{\set}[1]{\{\,{#1}\,\}}

\newcommand{\compose}{\circ}

\newcommand{\elesub}{\prec}

\newcommand{\Add}{\mathop{\rm Add}}

\newcommand{\Coll}{\mathop{\rm Coll}}

\newcommand{\plus}{{+}}

\newcommand{\restrict}{\upharpoonright} % uses amssymb
\newcommand{\satisfies}{\models}

\newcommand{\Union}{\bigcup}

\newcommand{\intersect}{\cap}

\newcommand{\LaverDiamond}{\mathop{\hbox{\line(0,1){10}\line(1,0){8}\line(-4,5){8}}\hskip 1pt}\nolimits}
\newcommand{\LD}{\LaverDiamond}

\newcommand{\smalllt}{\mathrel{\mathchoice{\raise2pt\hbox{$\scriptstyle<$}}{\raise1pt\hbox{$\scriptstyle<$}}{\raise0pt\hbox{$\scriptscriptstyle<$}}{\scriptscriptstyle<}}}
\newcommand{\smallleq}{\mathrel{\mathchoice{\raise2pt\hbox{$\scriptstyle\leq$}}{\raise1pt\hbox{$\scriptstyle\leq$}}{\raise1pt\hbox{$\scriptscriptstyle\leq$}}{\scriptscriptstyle\leq}}}
\newcommand{\lt}{\smalllt}

\newcommand{\ltkappa}{{{\smalllt}\kappa}}

\newcommand{\ltlambda}{{{\smalllt}\lambda}}
\newcommand{\ltgamma}{{{\smalllt}\gamma}}

\newcommand{\ltdelta}{{{\smalllt}\delta}}

\newcommand{\boolval}[1]{\mathopen{\lbrack\!\lbrack}\,#1\,\mathclose{\rbrack\!\rbrack}}
\def\[#1]{\boolval{#1}}
\newbox\gnBoxA
\newdimen\gnCornerHgt
\setbox\gnBoxA=\hbox{\tiny$\ulcorner$}
\global\gnCornerHgt=\ht\gnBoxA
\newdimen\gnArgHgt
\def\gcode #1{%
\setbox\gnBoxA=\hbox{$#1$}%
\gnArgHgt=\ht\gnBoxA%
\ifnum     \gnArgHgt<\gnCornerHgt \gnArgHgt=0pt%
\else \advance \gnArgHgt by -\gnCornerHgt%
\fi \raise\gnArgHgt\hbox{\tiny$\ulcorner$} \box\gnBoxA %
\raise\gnArgHgt\hbox{\tiny$\urcorner$}}
\newcommand{\UnderTilde}[1]{{\setbox1=\hbox{$#1$}\baselineskip=0pt\vtop{\hbox{$#1$}\hbox to\wd1{\hfil$\sim$\hfil}}}{}}
\newcommand{\Undertilde}[1]{{\setbox1=\hbox{$#1$}\baselineskip=0pt\vtop{\hbox{$#1$}\hbox to\wd1{\hfil$\scriptstyle\sim$\hfil}}}{}}
\newcommand{\undertilde}[1]{{\setbox1=\hbox{$#1$}\baselineskip=0pt\vtop{\hbox{$#1$}\hbox to\wd1{\hfil$\scriptscriptstyle\sim$\hfil}}}{}}
\newcommand{\UnderdTilde}[1]{{\setbox1=\hbox{$#1$}\baselineskip=0pt\vtop{\hbox{$#1$}\hbox to\wd1{\hfil$\approx$\hfil}}}{}}
\newcommand{\Underdtilde}[1]{{\setbox1=\hbox{$#1$}\baselineskip=0pt\vtop{\hbox{$#1$}\hbox to\wd1{\hfil\scriptsize$\approx$\hfil}}}{}}

\newcommand{\st}{\mid}
\renewcommand{\th}{{\hbox{\scriptsize th}}}
\renewcommand{\implies}{\mathrel{\rightarrow}}

\newcommand{\iso}{\cong}
\def\<#1>{\langle#1\rangle}

\newcommand{\REG}{\mathop{{\rm REG}}}

\newcommand{\ZFC}{{\rm ZFC}}

\newcommand{\CH}{{\rm CH}}

\newcommand{\MA}{{\rm MA}}

\newcommand{\RA}{{\rm RA}}

\newcommand{\MM}{{\rm MM}}

\newcommand{\PFA}{{\rm PFA}}

\newcommand{\SPFA}{{\rm SPFA}}

\newcommand{\HOD}{{\rm HOD}}
\newcommand{\OD}{{\rm OD}}

\newcommand{\cell}[1]{\boxit{\hbox to 17pt{\strut\hfil$#1$\hfil}}}
\newcommand{\head}[2]{\lower2pt\vbox{\hbox{\strut\footnotesize\it\hskip3pt#2}\boxit{\cell#1}}}
\newcommand{\boxit}[1]{\setbox4=\hbox{\kern2pt#1\kern2pt}\hbox{\vrule\vbox{\hrule\kern2pt\box4\kern2pt\hrule}\vrule}}
\newcommand{\Col}[3]{\hbox{\vbox{\baselineskip=0pt\parskip=0pt\cell#1\cell#2\cell#3}}}
\newcommand{\tapenames}{\raise 5pt\vbox to .7in{\hbox to .8in{\it\hfill input: \strut}\vfill\hbox to
.8in{\it\hfill scratch: \strut}\vfill\hbox to .8in{\it\hfill output: \strut}}}
\newcommand{\Head}[4]{\lower2pt\vbox{\hbox to25pt{\strut\footnotesize\it\hfill#4\hfill}\boxit{\Col#1#2#3}}}
\newcommand{\Dots}{\raise 5pt\vbox to .7in{\hbox{\ $\cdots$\strut}\vfill\hbox{\ $\cdots$\strut}\vfill\hbox{\
$\cdots$\strut}}}
\newcommand{\df}{\it} % use italic for definition terms. Idea: also use this to create an index of definitions, if MakeIndex is true.